\documentclass[12pt, reqno]{amsart}
\usepackage{amssymb,latexsym,amsmath,amscd,amsthm,graphicx, color}
\usepackage[all]{xy}
\usepackage{pgf,tikz}
\usepackage{mathrsfs}
\usetikzlibrary{arrows}
\usepackage[left=0.6 in, top=0.6 in, right=0.6 in, bottom=0.3 in]{geometry}
\allowdisplaybreaks
\makeatletter
\newcommand{\setword}[2]{%
  \phantomsection
  #1\def\@currentlabel{\unexpanded{#1}}\label{#2}%
}
\makeatother

\usepackage[linkcolor=blue, urlcolor=blue, citecolor=blue,
colorlinks, bookmarks]{hyperref}

\definecolor{uuuuuu}{rgb}{0.26666666666666666,0.26666666666666666,0.26666666666666666}
\definecolor{xdxdff}{rgb}{0.49019607843137253,0.49019607843137253,1.}
\definecolor{ffqqqq}{rgb}{1.,0.,0.}
\definecolor{ffqqqq}{rgb}{1.,0.,0.}
\definecolor{ffxfqq}{rgb}{1.,0.4980392156862745,0.}

\definecolor{uuuuuu}{rgb}{0.26666666666666666,0.26666666666666666,0.26666666666666666}
\definecolor{qqwuqq}{rgb}{0.,0.39215686274509803,0.}
\definecolor{zzttqq}{rgb}{0.6,0.2,0.}
\definecolor{xdxdff}{rgb}{0.49019607843137253,0.49019607843137253,1.}
\definecolor{qqqqff}{rgb}{0.,0.,1.}
\definecolor{cqcqcq}{rgb}{0.7529411764705882,0.7529411764705882,0.7529411764705882}
\definecolor{sqsqsq}{rgb}{0.12549019607843137,0.12549019607843137,0.12549019607843137}

\theoremstyle{plain}

\newtheorem{theorem}[subsection]{Theorem}
\newtheorem{theo1}[subsubsection]{Theorem}

\newtheorem{corollary}[subsection]{Corollary}
\newtheorem{lemma}[subsection]{Lemma}
\newtheorem{defi}[subsection]{Definition}

\newtheorem{prop}[subsection]{Proposition}

\newtheorem{lemma1}[subsubsection]{Lemma}

\theoremstyle{definition}
\newtheorem{defi1}[subsection]{Definition}

\newtheorem{exam}[subsection]{Example}

\newtheorem{remark}[subsection]{Remark}
\newtheorem{remark1}[subsubsection]{Remark}

\newcommand{\uu}{\cup}
\newcommand{\ii}{\cap}
\newcommand{\UU}{\bigcup}

\newcommand{\ci}{\subseteq}
\newcommand{\sci}{\subset}

\newcommand{\es}{\emptyset}
\newcommand{\set}[1]{\{#1\}}

\newcommand{\ga}{\alpha}
\newcommand{\gb}{\beta}

\newcommand{\gd}{\delta}
\renewcommand{\gg}{\gamma}

\newcommand{\gk}{\kappa}

\newcommand{\gs}{\sigma}
\newcommand{\gt}{\tau}

\newcommand{\gL}{\Lambda}

\newcommand{\tit}{\textit}

\newcommand{\C}[1]{\mathcal{#1}}
\newcommand{\D}[1]{\mathbb{#1}}

\newcommand{\te}{\text}

\begin{document}
To appear, Journal of Fractal Geometry
\title{Constrained quantization for the Cantor distribution}

\author{$^1$Megha Pandey}
 \author{$^2$Mrinal Kanti Roychowdhury}

\address{$^{1}$Department of Mathematical Sciences \\
Indian Institute of Technology (Banaras Hindu University)\\
Varanasi, 221005, India.}
\address{$^{2}$School of Mathematical and Statistical Sciences\\
University of Texas Rio Grande Valley\\
1201 West University Drive\\
Edinburg, TX 78539-2999, USA.}

\email{$^1$meghapandey1071996@gmail.com, $^2$mrinal.roychowdhury@utrgv.edu}


\date{}
\maketitle

\pagestyle{myheadings}\markboth{Megha Pandey and Mrinal Kanti Roychowdhury}{Constrained quantization for the Cantor distribution}

\begin{abstract}
The theory of constrained quantization has been recently introduced by Pandey and Roychowdhury. In this paper, they have further generalized their previous definition of constrained quantization and studied the constrained quantization for the classical Cantor distribution. Toward this, they have calculated the optimal sets of $n$-points, $n$th constrained quantization errors, the constrained quantization dimensions, and the constrained quantization coefficients, taking different families of constraints for all $n\in \mathbb N$. The results in this paper show that both the constrained quantization dimension and the constrained quantization coefficient for the Cantor distribution depend on the underlying constraints. It also shows that the constrained quantization coefficient for the Cantor distribution can exist and be equal to the constrained quantization dimension. These facts are not true in the unconstrained quantization for the Cantor distribution. 
\end{abstract}

\section{Introduction}
Real-life problems, such as information theory, data compression, signal processing, etc., consist of a large number of data that is not easy to handle. In order to deal with such a data set, the theory of quantization comes into play (see, for instance, \cite{GG, GL, GL1, GN,  P, Z1, Z2}). Quantization is a process of discretization, in other words, to represent a set with a large number of elements, discrete or continuous, by a set with a smaller number of elements. Several mathematical theories have been introduced in the literature concerning the process of quantization. Graf and Luschgy gave the rigorous mathematical treatment in \cite{GL}. In \cite{GL2}, Graf and Luschgy studied the quantization problem for the canonical probability measure on the classical Cantor set. 
\par
Recently, in \cite{PR}, the authors introduced the concept of constrained quantization. A quantization without a constraint is referred to as an unconstrained quantization, which traditionally in the literature is referred to as a quantization, as mentioned in the previous paragraph. 
The theory of constrained quantization is a fascinating area of research, and it invites a lot of new areas to work with a number of applications. With the introduction of constrained quantization, quantization now has two classifications: \tit{constrained quantization} and \tit{unconstrained quantization}. 
In this paper, the authors have further generalized the definition of constrained quantization given in \cite{PR}, and study the concept of constrained quantization for the canonical probability measure on the classical Cantor set.

\begin{defi}\label{Vr}
Let $P$ be a Borel probability measure on space $\D R^k$ equipped with a metric $d$ induced by a norm $\lVert\cdot\rVert$ on $\D R^k$, and $r \in (0, \infty)$. 
Let $\set{S_j\ci \D R^k: j\in \D N}$, where $\mathbb N$ denotes the set of all natural numbers, be a family of closed sets with $S_1$ nonempty. Then, for $n\in \mathbb{N}$, the \tit {$n$th constrained quantization
error} for $P$, of order $r$ with respect to the family of constraints $\set{S_j\ci \D R^k: j\in \D N}$, is defined by
\begin{equation} \label{EqVr}
V_{n, r}:=V_{n, r}(P)=\inf \Big\{\int \mathop{\min}\limits_{a\in\ga} d(x, a)^r dP(x) : \ga \ci \UU_{j=1}^nS_j, ~ 1\leq  \text{card}(\ga) \leq n \Big\},
\end{equation}
where $\te{card}(A)$ represents the cardinality of the set $A$.
\end{defi}
The number 
\begin{equation*}
  V_r(P; \ga):= \int \mathop{\min}\limits_{a\in\ga} d(x, a)^r dP(x)
\end{equation*}
is called the distortion error for $P$, of order $r$, with respect to a set $\ga \ci \D R^k$. 
The sets $S_j$ are the constraints in the constrained quantization error.
We assume that $\int d(x, 0)^r dP(x)<\infty$ to make sure that the infimum in \eqref{EqVr} exists (see \cite{PR}).  A set $ \ga \ci \mathop{\UU}\limits_{j=1}^n S_j$ for which the infimum in  \eqref{EqVr} exists and does not contain more than $n$ elements is called an \tit{optimal set of $n$-points} for $P$. Elements of an optimal set are called \tit{optimal elements}. 
\begin{remark}
In Definition~\ref{Vr} of constrained quantization error, if all $S_j$ for $j\in \D N$ are identical, then it reduces to the definition of constrained quantization error introduced by Pandey and Roychowdhury in \cite{PR}. Furthermore, if
$S_j=\D R^k$ for all $j\in \D N$, then it reduces to the definition of \tit{$n$th unconstrained quantization error}, which traditionally in the literature is referred to as the \tit{$n$th quantization error} (see \cite{GL}).  
For some recent work in the direction of unconstrained quantization, one can see \cite{GL,  GL2, DFG, DR, GL3, KNZ, MR, PRRSS, P1, R1, R2, R3, R4}. 
\end{remark} 

 Let $V_{n, r}(P)$ be a strictly decreasing sequence, and write $V_{\infty, r}(P):=\mathop{\lim}\limits_{n\to \infty} V_{n, r}(P)$. 
Then, the number $D_r(P)$ defined by 
\begin{equation} \label{eq001000} D_r(P):=\mathop{\lim}\limits_{n\to \infty} \frac{r\log n}{-\log (V_{n, r}(P)-V_{\infty, r}(P))}, 
\end{equation} 
if it exists, is called the \tit{constrained quantization dimension} of $P$ of order $r$. The constrained quantization dimension measures the speed at which the specified measure of the constrained quantization error converges as $n$ tends to infinity.
For any $\gk>0$, the  number  
\begin{equation} \label{eq00100} \lim_{n\to \infty} n^{\frac r \gk}  (V_{n, r}(P)-V_{\infty, r}(P)),\end{equation} if it exists, is called the \tit{$\gk$-dimensional constrained quantization coefficient} for $P$ of order $r$.

\begin{remark}
In unconstrained quantization, $V_{\infty, r}(P):=\mathop{\lim}\limits_{n\to \infty} V_{n, r}(P)=0$. Hence, in unconstrained quantization, i.e., when $V_{\infty, r}(P)=0$, the definitions of constrained quantization dimension and the $\gk$-dimensional constrained quantization coefficient defined by \eqref{eq001000} and \eqref{eq00100}, respectively, reduce to the corresponding definitions in unconstrained scenario (see \cite{GL}). 
\end{remark}

This paper deals with $r=2$ and $k=2$, and the metric on $\D R^2$ as the Euclidean metric induced by the Euclidean norm $\lVert \cdot\rVert$. Instead of writing $V_r(P; \ga)$ and $V_{n, r}:=V_{n, r}(P)$ we will write them as $V(P; \ga)$ and $V_n:=V_{n}(P)$, i.e., $r$ is omitted in the subscript as $r=2$ throughout the paper.
Let us take the family $\set{S_j : j\in \D N}$, that occurs in Definition \ref{Vr} as follows:
\begin{equation} \label{eq000} S_j=\set{(x, y) : 0\leq x\leq 1 \te{ and } y=\frac 1j}  \te{ for all } j\in \D N.
\end{equation}
Let $T_1, T_2 : \mathbb R \to \mathbb R$ be two contractive similarity mappings such that
$T_1(x)=\frac 13 x$ and $ T_2 (x)=\frac 13 x +\frac 23$. Then, there exists a unique Borel probability measure $P$
on $\D R$ such that
$P=\frac 12 P\circ T_1^{-1}+\frac 12 P\circ T_2^{-1}$, where $P\circ T_i^{-1}$ denote the image measures of $P$ with respect to
$T_i$ for $i=1,  2$ (see \cite{H}). If $k\in \D N$, and $\gs:=\gs_1\gs_2 \cdots \gs_k \in \{ 1, 2\}^k$, then we call $\gs$ a word of length $k$ over the alphabet $I:=\set{1, 2}$, and denote it by
$|\gs|:=k$. By $I^\ast$, we denote the set of all words, including the empty word $\es$. Notice that the empty word has a length zero. For any word $\gs:=\gs_1\gs_2 \cdots \gs_k \in I^\ast$,  we write \[T_\gs:=T_{\gs_1}\circ \cdots \circ T_{\gs_k} \text{ and } J_\gs:=T_\gs([0, 1]).\]
Then, the set $C:=\bigcap_{k\in \mathbb N} \bigcup_{\gs \in \{1, 2\}^k} J_\gs$ is known as the \textit{Cantor set} generated by the
two mappings $T_1$ and $T_2$, and equals the support of the probability measure $P$, where $P$ can be written as
 \[P=\sum_{\gs\in \set{1, 2}^k} \frac 1 {2^k}  P\circ T_\gs^{-1}.\]
For this probability measure $P$, Graf and Luschgy determined the optimal sets of $n$-means and the $n$th quantization errors for all $n\in \D N$ (see \cite{GL2}). They also showed that the unconstrained quantization dimension of the measure $P$ exists and equals $\frac{\log 2}{\log 3}$, which is the Hausdorff dimension of the Cantor set $C$, and the unconstrained quantization coefficient does not exist. In fact, in \cite{GL2}, they showed that the lower and the upper unconstrained quantization coefficients exist as finite positive numbers.

\subsection{Delineation} 
In this paper, first, we have determined the optimal sets of $n$-points and the $n$th constrained quantization errors for all $n\in\D N$ for the Borel probability measure $P$ with support the Cantor set $C$. Then, we have calculated the constrained quantization dimension and the constrained quantization coefficient. We have shown that both the constrained quantization dimension $D(P)$ and the $D(P)$-dimensional constrained quantization coefficient exist and are equal to one, i.e., they coincide. 
Then, in the last section, taking different families of constraints for all $n\in \D N$, we investigate the optimal sets of $n$-points, $n$th constrained quantization errors, the constrained quantization dimensions, and the constrained quantization coefficients. 
From work in this paper, it can be seen that the constrained quantization dimension of the Cantor distribution depends on the family of constraints $\set{S_j: j\in \D N}$, i.e., the constrained quantization dimension is not always equal to the Hausdorff dimension of the Cantor set as it occurs in the case of unconstrained quantization (see \cite{GL2}). In the unconstrained quantization, the $D(P)$-dimensional quantization coefficient does not exist (see \cite{GL2}). But from work in the last section, we see that the $D(P)$-dimensional constrained quantization coefficient also depends on the constraints; it may or may not exist.
 
\section{Preliminaries}
In this section, we give some basic notations and definitions which we have used throughout the paper. 
As defined in the previous section, let $I:=\{1, 2\} $ be an alphabet. For any two words $\gs:=\gs_1\gs_2\cdots \gs_k$ and
$\tau:=\tau_1\tau_2\cdots \tau_\ell$ in $I^*$, by
$\gs\tau:=\gs_1\cdots \gs_k\tau_1\cdots \tau_\ell$, we mean the word obtained from the concatenation of the two words $\gs$ and $\tau$. For $\gs, \gt\in I^\ast$, $\gs$ is called \textit{an extension of} $\gt$ if $\gs=\gt x$ for some word $x\in I^\ast$. The mappings $T_i:\D R \to \D R,\ i=1,2, $ such that $T_1(x)=\frac 13x$ and $T_2(x)=\frac 13x+\frac 23$ are the generating maps of the Cantor set $C $, which is the support of the probability measure $P$ on $\D R$ given by $P=\frac 12 P\circ T_1^{-1}+\frac 12 P\circ T_2^{-1}$. For $\gs:=\gs_1\gs_2 \cdots\gs_k \in I^k$, write 
$J_\gs=T_\gs [0, 1]$, where $T_\gs:=T_{\gs_1}\circ T_{\gs_2}\circ\cdots \circ T_{\gs_k}$ is a composition mapping. Notice that $J:=J_\es=T_\es[0, 1]=[0,1]$. 
Then, for any $k\in \D N$, as mentioned before, we have 
\[C=\bigcap_{k\in \mathbb N} \bigcup_{\gs \in I^k} J_\gs \te{ and  } P=\sum_{\gs \in I^k}\frac 1 {2^k} P\circ T_\gs^{-1}.\]
The elements of the set $\{J_\gs : \gs \in I^k \}$ are the $2^k$ intervals in the $k$th level in the construction of the Cantor set $C$, and are known as the \textit{basic intervals at the $k$th level.}  The intervals $J_{\gs 1}$, $J_{\gs 2}$, into which $J_\gs$ is split up at the $(k+1)$th level are called the \textit{children of $J_\gs$}.

With respect to a finite set $\ga \sci \D R^2$, by the \tit{Voronoi region} of an element $a\in \ga$, it is meant the set of all elements in $\D R^2$ which are nearest to $a$ among all the elements in $\ga$, and is denoted by $M(a|\ga)$.
 For any two elements $(a, b)$ and $(c, d)$ in $\D R^2$, we write 
 \[\rho((a, b), (c, d)):=(a-c)^2 +(b-d)^2,\] which gives the squared Euclidean distance between the two elements $(a, b)$ and $(c, d)$. Let $p$ and $q$ be two elements that belong to an optimal set of $n$-points for some positive integer $n$. Then, $p$ and $q$  are called \tit{adjacent elements} if they have a common boundary in their own Voronoi regions. Let $e$ be an element on the common boundary of the Voronoi regions of the adjacent elements $p$ and $q$.
 Since the common boundary of the Voronoi regions of any two adjacent elements is the perpendicular bisector of the line segment joining the elements, we have
\[\rho(p, e)-\rho(q, e)=0. \]
We call such an equation a \tit{canonical equation}.
Notice that any element $x\in \D R$ can be identified as an element $(x, 0)\in \D R^2$. Thus, the nonnegative real-valued function $\rho$ on $\D R \times \D R^2$ defined by 
\[\rho: \D R \times \D R^2 \to [0, \infty) \te{ such that } \rho(x, (a, b))=(x-a)^2 +b^2,\]
represents the squared Euclidean distance between an element $x\in \D R$ and an element $(a, b)\in \D R^2$. Let $\pi: \D R^2 \to \D R$ such that $\pi(a, b)=a$ for any $(a, b) \in \D R^2$ denote the projection mapping. For a random variable $X$ with distribution $P$, let $E(X)$ represent the expected value, and $V:=V(X)$ represent the variance of $X$. 

The following lemmas are well-known (see \cite{GL2}).  
\begin{lemma} \label{lemma1}
Let $f : \mathbb R \to \mathbb R^+$ be Borel measurable and $k\in \mathbb N$. Then,
\[\int f dP=\sum_{\gs \in \{1, 2\}^k} p_\gs \int f \circ T_\gs dP.\]
\end{lemma}

 \begin{lemma} \label{lemma2}
Let $ X$ be a random variable with probability distribution $P$.  Then, $E(X)= \frac 12  \te{ and } V:=V(X)=E\lVert X-\frac 1 2 \rVert^2=E(X-\frac 12)^2=\frac 1 8$. Moreover, for any $x_0\in \D R$, we have 
\[\int (x-x_0)^2 dP(x)=V(X)+(x-\frac 12)^2.\]
\end{lemma}

\begin{remark}
For words $\gb, \gg, \cdots, \gd$ in $I^\ast$, by $a(\gb, \gg, \cdots, \gd)$ we mean the conditional expectation of the random variable $ X$ given $J_\gb\uu J_\gg \uu\cdots \uu J_\gd,$ i.e.,
\begin{equation*} \label{eq45} a(\gb, \gg, \cdots, \gd)=E(X : X\in J_\gb \uu J_\gg \uu \cdots \uu J_\gd)=\frac{1}{P(J_\gb\uu \cdots \uu J_\gd)}\int_{J_\gb\uu \cdots \uu J_\gd} x \, dP.
\end{equation*}
Recall Lemma \ref{lemma1}, for each $\gs \in I^\ast$, since $T_\gs$ is a similarity mapping, we have
\begin{align*}
a(\gs)&=E(X : X \in J_\gs) =\frac{1}{P(J_\gs)} \int_{J_\gs} x \,dP=\int_{J_\gs} x  d(P\circ T_\gs^{-1})=\int T_\gs (x) \, dP\\
&=E(T_\gs(X))=T_\gs(E(X))=T_\gs(\frac 12).
\end{align*}
\end{remark}

\begin{defi1}  \label{defi1} 
For $n\in \D N$ with $n\geq 2$, let $\ell(n)$ be the unique natural number with $2^{\ell(n)} \leq n<2^{\ell(n)+1}$ and $S_n=\set{(x, \frac 1 n) : 0\leq x\leq 1}$. For $I\sci \set{1, 2}^{\ell(n)}$ with card$(I)=n-2^{\ell(n)}$ let $\ga_n(I)\ci S_n$ be the set such that
\[\ga_n(I)=\set{(a(\gs), \frac 1{n}) : \gs \in \set{1,2}^{\ell(n)} \setminus I} \uu \set{(a(\gs 1), \frac 1 n) : \gs \in I} \uu \set {(a(\gs 2), \frac  1n) : \gs \in I}.\]
\end{defi1}

\begin{prop} \label{prop0}
Let $\ga_n(I)$ be the set given by Definition \ref{defi1}. Then, the number of such sets is  ${}^{2^{\ell(n)}}C_{n-2^{\ell(n)}}$, and the corresponding distortion error is given by
\begin{equation*}\label{eq11}
V(P; \ga_n(I))=\int\mathop{\min}\limits_{a\in\ga_n(I)} \rho(x, a)\, dP=\frac 1 {18^{\ell(n)}}V\Big(2^{\ell(n)+1}-n+\frac 1 9(n-2^{\ell(n)})\Big)+\frac 1 {n^2},
\end{equation*}
where $V$ is the variance as stated in Lemma~\ref{lemma2}.
\end{prop}
\begin{proof} 
If $2^{\ell(n)} \leq n<2^{\ell(n)+1}$, then the subset $I$ can be chosen in ${}^{2^{\ell(n)}}C_{n-2^{\ell(n)}}$ different ways, and so, the number of such sets is given by ${}^{2^{\ell(n)}}C_{n-2^{\ell(n)}}$, and the corresponding distortion error is obtained as
\begin{align*}
V(P; \ga_n(I))=&\int \min_{a\in \ga_n(I)}\rho(x, a)\, dP\\
=&\sum_{\gs \in \set{1, 2}^{\ell(n)}\setminus I}\int_{J_\gs}\rho(x, (a(\gs), \frac 1 n))\,dP\\
&+\sum_{\gs \in I} \Big(\int_{J_{\gs1}}\rho(x, (a(\gs1), \frac 1 n))\,dP +\int_{J_{\gs2}}\rho(x, (a(\gs2), \frac 1 n))\,dP\Big)\\
=&\sum_{\gs \in \set{1, 2}^{\ell(n)}\setminus I}\frac 1 {2^{\ell(n)}} \int\rho(T_\gs(x), (a(\gs), \frac 1 n))\,dP\\
&+\sum_{\gs \in I}\frac 1{2^{\ell(n)+1}} \Big(\int\rho(T_{\gs1}(x), (a(\gs1), \frac 1 n))\,dP\\
&+\int\rho(T_{\gs2}(x), (a(\gs2), \frac 1 n))\,dP\Big)\\
=&\sum_{\gs \in \set{1, 2}^{\ell(n)}\setminus I}\frac 1 {2^{\ell(n)}}\Big(\frac 1 {9^{\ell(n)}}V +\frac 1{n^2})+\sum_{\gs \in I}\frac 1{2^{\ell(n)+1}} \Big(\frac 2{9^{\ell(n)+1}} V+\frac 2{n^2}\Big)\\
=&\frac 1 {18^{\ell(n)}}V\Big(2^{\ell(n)+1}-n+\frac 1 9(n-2^{\ell(n)})\Big)+\frac 1 {n^2}.
 \end{align*}
 Thus, the proof of the proposition is complete.
\end{proof}

In the next sections, we give the main results of the paper. 

\section{Optimal sets of $n$-points for all $n\geq 1$} \label{sec1}
In this section, we calculate the optimal sets of $n$-points and the $n$th constrained quantization errors for all $n\in \D N$. For $j\in \D N$, we have the constraints as
\[S_j=\set{(x, y) : 0\leq x\leq 1 \te{ and } y=\frac 1j} \te{ for all } j\in \D N.\]
For all $j\in \D N$, the perpendiculars on the constraints $S_j$ passing through the points $(a, \frac 1 j)\in S_j$ intersect $J$ at the points $a$, where $0\leq a\leq 1$. Thus, for each $j\in \D N$, there exists a one-one correspondence between the element $(a, \frac 1 j)$ on $S_j$ and the element $a$ on $J$. Thus, for $j\in \D N$, there exist bijective functions
\begin{equation} \label{sec30} U_j : S_j \to J\te{ such that } U_j(a, \frac 1j)=a.\end{equation}
Hence, the inverse functions $U_j^{-1}$ are defined as
\[U_j^{-1} : J\to S_j \te{ such that } U_j^{-1}(x)=(x, \frac 1 j).\]

\begin{remark} \label{rem3.0}
For $n\geq 2$, let $\ga_n(I)$ be the set given by Definition \ref{defi1}, and for each $j\in \mathbb{N}$, let $U_j$ be the bijective functions defined by \eqref{sec30}. Then, Proposition \ref{prop0} implies that
\[V(P; \ga_n(I))=V(P; U_n(\ga_n(I)))+\frac 1 {n^2}.\]
\end{remark}

\begin{prop}
An optimal set of one-point is $\set{(\frac 12, 1)}$ with constrained quantization error $V_1=\frac{9}{8}$.
\end{prop}
\begin{proof}
Let $\ga:=\set{(a, b)}$ be an optimal set of one-point. Since $\ga \ci S_1$, we have $b=1$. Now, the distortion error for $P$ with respect to the set $\ga$ is give by
\[V(P; \ga)=\int \rho(x, (a, 1)) dP=a^2-a+\frac{11}{8},\]
the minimum value of which is $\frac{9}{8}$ and it occurs when $a=\frac 12$. Thus, the proof of the proposition is complete.
\end{proof}

The following lemma plays an important role in the paper.

\begin{lemma} \label{lemma0}
Let $\ga_n\ci \mathop{\uu}\limits_{j=1}^n S_j$ be an optimal set of $n$-points for $P$ such that
\[\ga_n:=\set{(a_j, b_j) : 1\leq j\leq n},\]
where $a_1<a_2<a_3<\cdots<a_n$. Then, $a_j=E(X :  X\in \pi(M((a_j, b_j)|\ga_n)))$ and $b_j=\frac 1 n$,
where $M((a_j, b_j)|\ga_n)$ are the Voronoi regions of the elements $(a_j, b_j)$ with respect to the set $\ga_n$ for $1\leq j\leq n$.
 \end{lemma}

\begin{proof}
Let $\ga_n:=\set{(a_j, b_j): 1\leq j\leq n}$, as given in the statement of the lemma, be an optimal set of $n$-points. Take any $(a_q, b_q)\in \ga_n$. Since $\ga_n\ci \mathop{\uu}\limits_{j=1}^n S_j$, we can assume that $(a_q, b_q) \in S_t$, i.e., $b_q=\frac 1 t$ for some $1\leq t\leq n$. Since the Voronoi region of $(a_q, b_q)$, i.e., $M((a_q, b_q)|\ga_n)$ has positive probability, $M((a_q, b_q)|\ga_n)$ contains some basic intervals from $J$ that generates the Cantor set $C$. Let $\set{J_{\gs^{(j)}} :  j\in \gL}$, where $\gL$ is an index set, be the family of all basic intervals that are contained in $M((a_q, b_q)|\ga_n)$. Now, the distortion error contributed by $(a_q, b_q)$ in its Voronoi region $M((a_q, b_q)|\ga_n)$ is given by
\begin{align*}
&\int_{M((a_q, b_q)|\ga_n)}\rho(x, (a_q, b_q)) \,dP=\sum_{j\in \gL} \frac 1{2^{\ell(\gs^{(j)})}} \int_{J_{\gs^{(j)}}}\rho(x, (a_q, b_q)) \,d(P\circ T_{\gs^{(j)}}^{-1})\\
&=\sum_{j\in \gL} \frac 1{2^{\ell(\gs^{(j)})}} \frac 1{9^{\ell(\gs^{(j)})}}V+\sum_{j\in \gL} \frac 1{2^{\ell(\gs^{(j)})}}\rho(T_{\gs^{(j)}} (\frac 12), (a_q, \frac 1 t))\\
&=\sum_{j\in \gL} \frac 1{2^{\ell(\gs^{(j)})}} \frac 1{9^{\ell(\gs^{(j)})}}V+\sum_{j\in \gL} \frac 1{2^{\ell(\gs^{(j)})}}\Big((T_{\gs^{(j)}} (\frac 12)-a_q)^2+ \frac 1{t^2}\Big)\\
&=\sum_{j\in \gL} \frac 1{2^{\ell(\gs^{(j)})}} \frac 1{9^{\ell(\gs^{(j)})}}V+\sum_{j\in \gL} \frac 1{2^{\ell(\gs^{(j)})}}(T_{\gs^{(j)}} (\frac 12)-a_q)^2+\sum_{j\in \gL} \frac 1{2^{\ell(\gs^{(j)})}} \frac 1{t^2}.
\end{align*}
The above expression is minimum if both $\sum_{j\in \gL} \frac 1{2^{\ell(\gs^{(j)})}}(T_{\gs^{(j)}} (\frac 12)-a_q)^2$ and \\$\sum_{j\in \gL} \frac 1{2^{\ell(\gs^{(j)})}}\frac 1{t^2}$ are minimum, i.e., when
\[a_q=\frac {\sum_{j\in \gL} \frac 1{2^{\ell(\gs^{(j)})}}T_{\gs^{(j)}} (\frac 12)}{\sum_{j\in \gL} \frac 1{2^{\ell(\gs^{(j)})}}}=E(X :X\in \pi(M((a_q, b_q)|\ga_n))) \te{ and } b_q=\frac 1 t=\frac 1 n.\]
Since $(a_q, b_q)\in \ga_n$ is chosen arbitrarily, the proof of the lemma is complete.  
\end{proof}

\begin{remark}\label{rem1}
Let $\ga_n\ci \mathop{\uu}\limits_{j=1}^n S_j$ be an optimal set of $n$-points for $P$ such that
\[\ga_n:=\set{(a_j, b_j) : 1\leq j\leq n},\]
where $a_1<a_2<a_3<\cdots<a_n$. Then, by Lemma \ref{lemma0}, we can deduce that $0<a_1<\cdots<a_n<1$.
\end{remark}

\begin{remark}\label{rem111}
Lemma \ref{lemma0} implies that if $\ga_n$ is an optimal set of $n$-points for $P$, then $\ga_n\ci S_n$ for all $n\in \D N$.
\end{remark}

\begin{prop} \label{prop31}
The set $\set{(\frac 16, \frac 12), (\frac 56, \frac 12)}$ forms an optimal set of two-points with constrained quantization error $V_2=\frac{19}{72}$.
\end{prop}
\begin{proof}
Due to symmetry, the distortion error due to the set $\gb:=\set{(\frac 16, \frac 12), (\frac 56, \frac 12)}$ is given by
\[V(P; \gb)= 2\int_{J_1}\rho(x, (\frac 16, \frac 12))\,dP=\frac{19}{72}.\]
Let $\ga:=\set{(a_1, \frac 12), (a_2, \frac 12)}$, where $0<a_1<a_2<1$, be an optimal set of two-points. Since $V_2$ is the constrained quantization error for two-points, we have $V_2\leq V(P; \gb)=\frac{19}{72}$. We first show that $U_2(\ga)\ii J_1\neq \es$. Suppose that $T_{21}(1)=\frac 79\leq a_1$. Then,
\[V_2>\int_{J_1}\rho(x, (\frac 79, \frac 12))\,dP=\frac{413}{1296}>V_2,\]
which leads to a contradiction. Suppose that $\frac 23\leq a_1<\frac 79$. Then, the Voronoi region of $(a_1, \frac 12)$ does not contain any element from $J_{22}$. For the sake of contradiction, assume that the Voronoi region of $(a_1, \frac 12)$ contains elements from $J_{22}$. Then, we must have $\frac 12(a_1+a_2)>\frac 89$ implying $a_2>\frac {16}9-a_1>\frac {16}9-\frac {7}9=1$, which gives a contradiction. Thus, we see that $J_{22}$ is contained in the Voronoi region of $(a_2, \frac 12)$. Hence,
\[V_2>\int_{J_1}\rho(x, (\frac 23, \frac 12))\,dP+\int_{J_{22}}\rho(x, (a(22), \frac 12))\,dP=\frac{829}{2592}>V_2,\]
which gives a contradiction. Assume that $\frac 13\leq a_1<\frac 23$. Then, the distortion error is obtained as
\[V_2\geq \int_{J_1}\rho(x, (\frac 13, \frac 12))\,dP+\int_{J_{21}}\rho(x, (\frac 23, \frac 12))\,dP+\int_{J_{22}}\rho(x, (a(22), \frac 12))\,dP=\frac{353}{1296}>V_2,\]
which yields a contradiction. Hence, we can assume that $a_1<\frac 13$, i.e., $U_2(\ga)\ii J_1\neq \es$. Similarly, we can show that $\frac 23<a_2$, i.e., $U_2(\ga)\ii J_2\neq \es$. Hence, $U_2(\ga)\ii J_j\neq \es$ for $j=1, 2$. Notice that then the Voronoi region of $(a_1, \frac 12)$ does not contain any element from $J_2$, and the Voronoi region of $(a_2, \frac 12)$ does not contain any element from $J_1$ yielding
\[(a_1,\frac 12)=(a(1), \frac 12)=(\frac 16, \frac 12), \te{ and } (a_2,\frac 12)=(a(2), \frac 12)=(\frac 56, \frac 12),\]
with constrained quantization error $V_2=\frac{19}{72}$. Thus, the proof of the proposition is complete.
\end{proof}

\begin{lemma} \label{lemma31}
Let $\ga_3$ be an optimal set of three-points. Then, $U_3(\ga_3)\ii J_j\neq \es$ for $j=1, 2$, and $U_3(\ga_3)$ does not contain any element from the open line segment joining $(\frac 13, 0)$ and $(\frac 23, 0)$.
\end{lemma}

\begin{proof}
The distortion error due to the set $\gb:=\set{(a(11), \frac 13), (a(12), \frac 13), (a(2), \frac 13)}$ is given by
\[V(P; \gb)=\int\min_{a\in \ga_3} \rho(x, a)\,dP=\frac{77}{648}.\]
Let $\ga_3:=\set{(a_1, \frac 13), (a_2, \frac 13), (a_3, \frac 13)}$ be an optimal set of three-points such that $0<a_1<a_2<a_3<1$. Since $V_3$ is the constrained quantization error for three-points, we have $V_3\leq \frac{77}{648}$. Let us first show that $U_3(\ga_3)\ii J_1\neq \es$, i.e., $\frac 13<a_1$. We prove it by contradiction. Notice that 
\[\int_{J_1}\rho(x, (a, \frac 13))>\frac{77}{648},\]
if $a>\frac{1}{36} \big(\sqrt{146}+6\big)$. Choose a number $\frac{211}{420}>\frac{1}{36} \big(\sqrt{146}+6\big)$, and consider the following cases: 

\tit{Case~1. $\frac{211}{420}\leq a_1$.}

 Then, \[V_3\geq \int_{J_1}\rho(x, (\frac{211}{420}, \frac 13)) \,dP=\frac{4659}{39200}>V_3,\]
 which gives a contradiction.

\tit{Case~2. $\frac{11}{27} \leq a_1<\frac{211}{420}$.}

Then, we have
$\frac 12(a_1+a_2)>\frac 23$ implying $a_2>\frac 43-a_1>\frac 43-\frac{211}{420}=\frac{349}{420}>T_{21} (1)$. Hence,
\[V_3\geq \int_{J_1}\rho(x, (\frac{11}{27}, \frac 13)) \,dP+\int_{J_{21}}\rho(x, (\frac{349}{420}, \frac 13)) \,dP=\frac{7006871}{57153600}>V_3,\]
which leads to a contradiction.

\tit{Case~3. $\frac 13< a_1<\frac{11}{27}$.}

Then, we have
$\frac 12(a_1+a_2)>\frac 23$ implying $a_2>\frac 43-a_1>\frac 43-\frac{11}{27}=\frac{25}{27}=T_{221} (1)$. Hence,
\[V_3\geq \int_{J_1}\rho(x, (\frac 13, \frac 13)) \,dP+\int_{J_{21}\uu J_{221}}\rho(x, (\frac{25}{27}, \frac 13)) \,dP=\frac{6013}{46656}>V_3,\]
which gives a contradiction.

Taking into account the above cases, we see that a contradiction arises. Hence, $U_3(\ga_3)\ii J_1\neq \es$. Reflecting the above arguments with respect to the line $x=\frac 12$, we can show that $U_3(\ga_3)\ii J_2\neq \es$. Thus, $U_3(\ga_3)\ii J_j\neq \es$ for $j=1, 2$. We now show that $U_3(\ga_3)$ does not contain any element from the open line segment joining $(\frac 13, 0)$ and $(\frac 23, 0)$. For the sake of contradiction, assume that $U_3(\ga_3)$ contains an element from the open line segment joining $(\frac 13, 0)$ and $(\frac 23, 0)$. Since $U_3(\ga_3)\ii J_j\neq \es$ for $j=1, 2$, we must have $\frac 13<a_2<\frac 23$. The following cases can happen:

\tit{Case~I. $\frac 12\leq a_2<\frac 23$.}

Then, we have $\frac 12(a_1+a_2)<\frac 13$ implying $a_1<\frac 23-a_2\leq \frac 16$. Again, $E( X : X\in J_1)=\frac 16$. Hence,
\begin{align*}
V_3&=\int_{J_1}\min_{a\in\set{a_1, a_2}}\rho(x, a)\,dP+\int_{J_{21}\uu J_{22}}\min_{a\in\set{a_2, a_3}}\rho(x, a)\,dP\\
&=\int_{J_1} \rho(x, (\frac 16, \frac 13))\,dP+\int_{J_{21}} \rho(x, (\frac 23, \frac 13))\,dP+\int_{J_{22}}\rho(x, (a(22), \frac 13))\,dP\\
&=\frac{155}{1296}>V_3,
\end{align*}
which yields a contradiction.

\tit{Case~II. $\frac 13<a_2<\frac12$.}

This case is the reflection of Case~I with respect to the line $x=\frac 12$. Hence, a contradiction also arises in this case.

Considering Case~I and Case~II, we can deduce that $U_3(\ga_3)$ does not contain any element from the open line segment joining $(\frac 13, 0)$ and $(\frac 23, 0)$.
Thus, the proof of the lemma is complete.
\end{proof}

\begin{prop} \label{prop32}
Let $\ga_n$ be an optimal set of $n$-points for all $n\geq 2$. Then, $U_n(\ga_n)\ii J_j\neq \es$ for $j=1, 2$, and $U_n(\ga_n)$ does not contain any element from the open line segment joining $(\frac 13, 0)$ and $(\frac 23, 0)$.
\end{prop}

\begin{proof}
For all $n\geq 2$, let us first prove that $U_n(\ga_n)\ii J_1\neq \es$. By Proposition \ref{prop31} and Lemma \ref{lemma31}, it is true for $n=2, 3$. Using a similar technique as Lemma \ref{lemma31}, we can also prove that it is true for any $n\geq 4$. However, here we give a general proof for all $n\geq 16$. The distortion error due to the set $\gb:=\set{(a(\gs), \frac 1{16}) : \gs \in \set{1,2}^4}$ is given by
\[V(P; \gb)=\int\min_{a\in\gb}\rho(x, a)\,dP=\frac{6593}{1679616}.\]
Since $V_n$ is the constrained quantization error for $n$-points with $n\geq 16$, and $V_n$ is a decreasing sequence, we have $V_n\leq V_{16}\leq \frac{6593}{1679616}$. Let $\ga_n:=\set{(a_j, \frac 1 n) : 1\leq j\leq n}$ be an optimal set of $n$-points such that $0<a_1<a_2<\cdots<a_n<1$.  For the sake of contradiction, assume that $U_n(\ga_n)\ii J_1= \es$. Then, $\frac 13<a_1$, and so,
\[V_n>\int_{J_1}\rho(x, (\frac 13, \frac 1 n)) \,dP=\frac{1}{2 n^2}+\frac{1}{48}\geq \frac{1}{48}>V_n,\]
which is a contradiction. Hence, we can assume that $U_n(\ga_n)\ii J_1\neq \es$. Similarly, we can show $U_n(\ga_n)\ii J_2\neq  \es$. Thus, the proof of the first part of the proposition is complete. We now show that $U_n(\ga_n)$ does not contain any element from the open line segment joining $(\frac 13, 0)$ and $(\frac 23, 0)$. For the sake of contradiction, assume that $U_n(\ga_n)$ contains an element from the open line segment joining $(\frac 13, 0)$ and $(\frac 23, 0)$.
Let
\[k=\min\set{j : a_j>\frac 13 \te{ for all } 1\leq j\leq n}.\]
Since $U_n(\ga_n)\ii J_1\neq \es$, we have $2\leq k$. Thus, we see that $a_{k-1}\leq \frac 13<a_k$. Again, recall that the Voronoi regions of the elements in an optimal set of $n$-points must have positive probability. Hence, we have $a_{k-1}\leq \frac 13<a_k<\frac 23\leq a_{k+1}$.

 The following cases can happen:

\tit{Case~1. $\frac 59 \leq a_k<\frac 23$.}

Then, $\frac 12(a_{k-1}+a_k)<\frac 13$ implying $a_{k-1}<\frac 23-a_k\leq \frac 23-\frac 59=\frac 19$. Hence,
 \[V_n\geq \int_{J_{12}}\rho(x, (\frac 19, \frac 1n)) \,dP=\frac{1}{4 n^2}+\frac{19}{2592}\geq \frac{19}{2592}>V_n,\]
 which gives a contradiction.

\tit{Case~2. $\frac 49<a_2<\frac 59$.}

Then,   $\frac 12(a_{k-1}+a_k)<\frac 13$ implying $a_{k-1}<\frac 23-a_k\leq \frac 23-\frac 49=\frac 29=T_{12} (0)$.
Again, $\frac 12(a_k+a_{k+1})>\frac 23$ implying $a_{k+1}>\frac 43-a_k>\frac 43-\frac 59=\frac79=T_{21} (1)$. Hence,
\[V_n\geq \int_{J_{12}}\rho(x, (\frac 29, \frac 1 n)) \,dP+\int_{J_{21}}\rho(x, (\frac 79, \frac 1n)) \,dP=\frac{1}{2 n^2}+\frac{11}{1296}\geq \frac{11}{1296}>V_n,\]
which leads to a contradiction.

\tit{Case~3. $\frac 13< a_k<\frac 49$.}

Then, $\frac 12(a_k+a_{k+1})>\frac 23$ implying $a_{k+1}>\frac 43-a_k>\frac 43-\frac 49=\frac 89$. Hence,
\[V_n\geq \int_{J_{21}}\rho(x, (\frac 89, \frac 1n)) \,dP=\frac{1}{4 n^2}+\frac{19}{2592}\geq \frac{19}{2592}>V_n,\]
 which gives a contradiction.

Considering the above all possible cases, we see that a contradiction arises. Hence, $U_n(\ga_n)$ does not contain any element from the open line segment joining $(\frac 13, 0)$ and $(\frac 23, 0)$. Thus, the proof of the proposition is complete.
\end{proof}

\begin{corollary} \label{cor32}
Let $n\geq 2$. Then, the Voronoi region of any element in $\ga_n\ii U_n^{-1}(J_1)$ does not contain any element from $J_2$, and the Voronoi region of any element in $\ga_n\ii U_n^{-1}(J_2)$ does not contain any element from $J_1$.
\end{corollary}
\begin{proof}
Let $\ga_n:=\set{(a_j,\frac 1n) : 1\leq j\leq n}$ be an optimal set of $n$-points such that $0<a_1<a_2<\cdots<a_n<1$. By Proposition \ref{prop32}, we see that $U_n(\ga_n)$ contains elements from both $J_1$ and $J_2$, and does not contain any element from the open line segment joining $(\frac 13, 0)$ and $(\frac 23, 0)$.
Let
\[k=\max\set{j : a_j\leq \frac 13 \te{ for all } 1\leq j\leq n}.\]
Then, $a_k\leq \frac 13<\frac 23\leq a_{k+1}$. For the sake of contradiction, assume that the Voronoi region $(a_k, \frac 1 n)$ contains an element from $J_2$. Then, $\frac 12(a_k+a_{k+1})\geq \frac 23$ yielding $a_{k+1}\geq \frac 43-a_k\geq \frac 43-\frac 1 3=1$, which is a contradiction. Similarly, we can show that if the Voronoi region $(a_{k+1}, \frac 1 n)$ contains an element from $J_1$, then a contradiction arises.
Thus, the proof of the corollary is complete.
\end{proof}

\begin{lemma} \label{lemma33}
For $n\geq 2$ let $\ga_n$ be an optimal set of $n$-points. Set $\gb_1:=U_n(\ga_n)\ii J_1$, $\gb_2:=U_n(\ga_n)\ii J_2$, and $n_1:=\te{card}(\gb_1)$. Then, $U_{n_1}^{-1}(T_1^{-1}(\gb_1))$ is an optimal set of $n_1$-points, $U_{n-n_1}^{-1}(T_2^{-1}(\gb_2))$ is an optimal set of $n_2:=(n-n_1)$-points, and
\[V_n= \frac 1{18}\Big(V_{n_1}+V_{n-n_1}-\frac 1 {n_1^2} - \frac 1{(n-n_1)^2}\Big)+\frac 1 {n^2}.\]
\end{lemma}
\begin{proof} By Proposition \ref{prop32}, we have  $U_n(\ga_n)=\gb_1\uu \gb_2$. Proceeding in a similar way as Lemma \ref{lemma0}, we can show that
\[V_n(P)=\int\min_{a\in\ga_n}\rho(x, a)\,dP=\int \min_{a\in U_n(\ga_n)} \rho(x, (a, 0)) \,dP+\frac 1 {n^2}.\]
Hence,
{\footnotesize\begin{align*} 
V_n=&\int_{J_1} \min_{a\in \gb_1} \rho(x, (a,0))\,dP+\int_{J_2} \min_{a\in \gb_2} \rho(x, (a, 0))\,dP+\frac 1 {n^2}\\
=&\frac 12 \int  \min_{a \in T_1^{-1}(\gb_1)} \rho(T_1(x),(T_1(a), 0))\,dP+\frac 12 \int \min_{a\in T_2^{-1}(\gb_2)} \rho(T_2(x), (T_2(a), 0))\,dP+\frac 1 {n^2}\\
=&\frac 1{18} \int  \min_{a\in T_1^{-1}(\gb_1)} \rho(x, (a, 0))\,dP+\frac 1{18} \int \min_{a\in T_2^{-1}(\gb_2)} \rho(x, (a, 0))\,dP+\frac 1 {n^2}\\
=&\frac 1{18} \Big(\int  \min_{a \in T_1^{-1}(\gb_1)} \rho(x, (a, \frac 1{n_1}))\,dP+ \int \min_{a\in T_2^{-1}(\gb_2)} \rho(x, (a, \frac 1 {n_2}))\,dP-\frac 1 {n_1^2}-\frac 1 {n_2^2}\Big)+\frac 1 {n^2}
\end{align*}}
implying
{\footnotesize\begin{equation} \label{eqim1}
V_n=\frac 1{18}\Big(\int  \min_{a \in U_{n_1}^{-1}T_1^{-1}(\gb_1)} \rho(x, a)\,dP+ \int \min_{a\in U_{n_2}^{-1}T_2^{-1}(\gb_2)} \rho(x, a)\,dP-\frac 1 {n_1^2}-\frac 1 {n_2^2}\Big)+\frac 1 {n^2}.
\end{equation}}
If $U_{n_1}^{-1}T_1^{-1}(\gb_1)$ is not an optimal set of $n_1$-points for $P$, then there exists a set $\gg_1\sci S_{n_1}$ with $\te{card}(\gg_1)=n_1$ such that
$\int \min_{a\in \gg_1} \rho(x, a)\, dP<\int \min_{a\in U_{n_1}^{-1}T_1^{-1}(\gb_1)} \rho(x, a)\, dP$. But then, $\gd:=T_1U_{n_1}(\gg_1)\uu \gb_2$ is a set of cardinality $n$. $V_n$ being the constrained quantization error for $n$-points, we have
\begin{equation} \label{eqim2}
V_n\leq \int\min_{a\in \gd} \rho(x, (a, \frac  1 n)\,dP=\int\min_{a\in \gd} \rho(x, (a, 0))\,dP+\frac 1 {n^2}.
\end{equation}
Notice that
\begin{align}\label{eqim3}
 &\int_{J_1} \min_{a\in T_1U_{n_1}(\gg_1)} \rho(x, (a,0))\,dP \notag\\
  =&\frac 1{18} \int  \min_{a \in U_{n_1}(\gg_1)} \rho(x, (a, 0))\,dP \notag
 =\frac 1{18} \Big(\int  \min_{a \in U_{n_1}(\gg_1)} \rho(x, (a, \frac 1{n_1}))\,dP-\frac 1{n_1^2}\Big)\notag\\
 =&\frac 1{18}\Big( \int  \min_{a \in \gg_1} \rho(x, a)\,dP-\frac 1{n_1^2}\Big) 
 <\frac 1{18}\Big(\int \min_{a\in U_{n_1}^{-1}T_1^{-1}(\gb_1)} \rho(x, a)\, dP-\frac 1{n_1^2}\Big).
\end{align}
Hence, by \eqref{eqim1}, \eqref{eqim2}, and \eqref{eqim3}, we have
\begin{align*}
V_n&\leq \int_{J_1}\min_{a\in T_1U_{n_1}(\gg_1)} \rho(x, (a, 0))\,dP+\int_{J_2}\min_{a\in\gb_2} \rho(x, (a, 0))\,dP+\frac 1 {n^2}\\
&<\frac 1{18}\Big(\int \min_{a\in U_{n_1}^{-1}T_1^{-1}(\gb_1)} \rho(x, a)\, dP+ \int \min_{a\in U_{n_2}^{-1}T_2^{-1}(\gb_2)} \rho(x, a)\,dP-\frac 1 {n_1^2}-\frac 1 {n_2^2}\Big)+\frac 1 {n^2}\\
&=V_n,
\end{align*}
which leads to a contradiction. Hence, $U_{n_1}^{-1}(T_1^{-1}(\gb_1))$ is an optimal set of $n_1$-points. Similarly, we see that $U_{n-n_1}^{-1}(T_2^{-1}(\gb_2))$ is an optimal set of $n_2:=(n-n_1)$-points, and hence,
\[V_n= \frac 1{18}\Big(V_{n_1}+V_{n-n_1}-\frac 1 {n_1^2} - \frac 1{(n-n_1)^2}\Big)+\frac 1 {n^2}.\]
Thus, the proof of the lemma is complete.
\end{proof}

In view of Lemma \ref{lemma33}, we give the following example.
\begin{exam}
Because of Proposition \ref{prop32} and Corollary \ref{cor32}, we can show that if $\ga_n$ is an optimal set of $n$-points with constrained quantization error $V_n$, then
\begin{align*}
\ga_3=&\set{(a(11), \frac 13), (a(12), \frac 13), (a(2), \frac 13)} \te{ with } V_3=\frac{127}{864},\\
\ga_4=&\set{(a(11), \frac 14), (a(12), \frac 14), (a(21), \frac 14), (a(22),\frac 14)} \te{ with } V_4=\frac{83}{1296},\\
\ga_7=&\set{(a(111), \frac 17), (a(112), \frac 17), (a(121), \frac 17), (a(122), \frac 17), (a(211), \frac 17),  (a(212), \frac 1 7),\\
 &(a(22), \frac 17)},
 \end{align*}
with $V_7=\frac{1993}{95256}$. Here
\begin{align*} \gb_1& =U_7(\ga_7)\ii J_1=\set{a(111), a(112), a(121), a(122)} \te{ and } \\
\gb_2&=U_7(\ga_7)\ii J_2=\set{a(211),  a(212), a(22)},
\end{align*}
with $\te{card}(\gb_1)=4$ and $\te{card}(\gb_2)=3$. Notice that
\begin{align*} U_4^{-1}(T_1^{-1}(\gb_1))&=\set{(a(11), \frac 14),  (a(12),\frac 14),  (a(21), \frac 14),  (a(22), \frac 14)}, \\
U_3^{-1}(T_2^{-1}(\gb_2))&=\set{(a(11), \frac 13), (a(12), \frac 13), (a(2), \frac 13)}, and \\
V_7&=\frac 1{18}(V_4+V_3-\frac 1 {4^2}-\frac 1{3^2})+\frac 1{7^2}.
\end{align*}
\end{exam}

Let us state and prove the following theorem, which gives the optimal sets of $n$-points for all $n\geq 2$.
\begin{theorem}  \label{theo1} 
 Let $P=\frac 12 P\circ T_1^{-1}+\frac 12 P\circ T_2^{-1}$ be a unique Borel probability measure on $\D R$ with support the
Cantor set $C$ generated by the two contractive similarity mappings $T_1(x)=\frac 13 x$ and $T_2(x)=\frac 13 x+\frac 23$ for all $x\in \D R$.
Then, the set $\ga_n:=\ga_n(I)$ given by Definition \ref{defi1} forms an optimal set of $n$-points for $P$ with the corresponding constrained quantization error $V_n:=V(P;\ga_n(I))$,
where $V(P; \ga_n(I))$ is given by Proposition \ref{prop0}.
\end{theorem}

\begin{proof}
We will proceed by induction on $\ell(n)$. If $\ell(n)=1$, then the theorem is true by Proposition \ref{prop31}. Let us assume that the theorem is true for all $\ell(n)<m$, where $m\in \D N$ and $m\geq 2$. We now show that the theorem is true if $\ell(n)=m$.  
Let $\ga_n:=\ga_n(I)$ be an optimal set of $n$-points for $P$ such that $2^m\leq n<2^{m+1}$. Set $\gb_1:=U_n(\ga_n)\ii J_1$, $\gb_2:=U_n(\ga_n)\ii J_2$,  $n_1:=\te{card}(\gb_1)$, and $n_2:=\te{card}(\gb_2)$. Then, by Lemma \ref{lemma33}, we have
\begin{equation} \label{eq341} V_n= \frac 1{18}\Big(V_{n_1}+V_{n_2}-\frac 1 {n_1^2} - \frac 1{n_2^2}\Big)+\frac 1 {n^2}.
\end{equation}
Without any loss of generality, we can assume that $n_1\geq n_2$. Let $p,  q\in \D N$ be such that
\begin{equation}\label{eq35} 2^p\leq n_1< 2^{p+1} \te{ and } 2^q\leq n_2< 2^{q+1}.
\end{equation}
We will show that $p=q=m-1$. Since $n_1\geq n_2$, we have $n_1\geq 2^{m-1}$ and $n_2< 2^{m}$. Hence, $p\geq m-1$ and $q\leq m-1$. 
If $\tilde V_n$ is the distortion error due to the set
\[\set{(a(\gs), \frac 1{n}) : \gs \in \set{1,2}^{m} \setminus I} \uu \set{(a(\gs 1), \frac 1 n) : \gs \in I} \uu \set {(a(\gs 2), \frac  1n) : \gs \in I},\]
where $I\sci \set{1, 2}^{m}$ with card$(I)=n-2^{m}$, then by Proposition \ref{prop0}, we have
\begin{align*}
\tilde V_n=\frac 1 {18^{m}}V\Big(2^{m+1}-n+\frac 1 9(n-2^{m})\Big)+\frac 1 {n^2}.
\end{align*}
Thus, by the induction hypothesis, we have  $\tilde V_n\geq V_n$, and then Equation~\eqref{eq341} implies that   
\begin{align*}
&\frac 1 {18^{m}}V\Big(2^{m+1}-n+\frac 1 9(n-2^{m})\Big)+\frac 1 {n^2} \geq \frac 1{18}\Big(V_{n_1}+V_{n-n_1}-\frac 1 {n_1^2} - \frac 1{(n-n_1)^2}\Big)+\frac 1 {n^2},
\end{align*}
i.e., 
\begin{align*}
\frac 1 {18^{m}}V\Big(2^{m+1}-n+\frac 1 9(n-2^{m})\Big) \geq&\frac 1{18^{m+1}}\Big(V\Big(2^{p+1}-n_1+\frac 1 9(n_1-2^{p})\Big)\\
&+\Big(2^{q+1}-n_2+\frac 1 9(n_2-2^{q})\Big)\Big),
\end{align*}
which is the same as Equation~5.9 in \cite{GL2}. Thus, proceeding in a similar way as \cite{GL2}, we have  $p=q=m-1$. By Lemma \ref{lemma33},
$U_{n_1}^{-1}(T_1^{-1}(\gb_1))$ is an optimal set of $n_1$-points, $U_{n-n_1}^{-1}(T_2^{-1}(\gb_2))$ is an optimal set of $n_2:=(n-n_1)$-points. Moreover, we have proved $2^{m-1}\leq n_1<2^{m}$, and $2^{m-1}\leq n_2<2^{m}$. Hence, by the induction hypothesis,
\begin{align*}
&U_{n_1}^{-1}(T_1^{-1}(\gb_1))\\
=&\set{(a(\gs), \frac 1{n}) : \gs \in \set{1,2}^{m-1} \setminus I_1} \uu \set{(a(\gs 1), \frac 1 n) : \gs \in I_1} \uu \set {(a(\gs 2), \frac  1n) : \gs \in I_1},
\end{align*}
where $I_1\sci \set{1, 2}^{m-1}$ with card$(I_1)=n_1-2^{m-1}$; 
and  
\begin{align*}
&U_{n_2}^{-1}(T_2^{-1}(\gb_2))\\
=&\set{(a(\gs), \frac 1{n}) : \gs \in \set{1,2}^{m-1} \setminus I_2} \uu \set{(a(\gs 1), \frac 1 n) : \gs \in I_2} \uu \set {(a(\gs 2), \frac  1n) : \gs \in I_2},
\end{align*}
where $I_2\sci \set{1, 2}^{m-1}$ with card$(I_2)=n_2-2^{m-1}$. Then, notice that 
\[\gb_1=\set{a(1\gs) : \gs \in \set{1,2}^{m-1} \setminus I_1} \uu \set{a(1\gs 1) : \gs \in I_1} \uu \set {a(1\gs 2) : \gs \in I_1},\]
 and  
\[\gb_2=\set{a(2\gs) : \gs \in \set{1,2}^{m-1} \setminus I_2} \uu \set{a(2\gs 1) : \gs \in I_2} \uu \set {a(2\gs 2) : \gs \in I_2}.\]
Take $I:=I_1\uu I_2$, and then 
\[\te{card}(I)=\te{card}(I_1)+\te{card}(I_2)=n_1-2^{m-1}+n_2-2^{m-1}=n-2^{m}.\] Hence, 
\begin{align*}
U_n(\ga_n)=\gb_1\uu \gb_2=\set{a(\gs) : \gs \in \set{1,2}^{m} \setminus I} \uu \set{a(\gs 1) : \gs \in I} \uu \set {a(\gs 2) : \gs \in I}. 
\end{align*} 
Thus, we have 
\[\ga_n:=\ga_n(I)=\set{(a(\gs), \frac 1n) : \gs \in \set{1,2}^{m} \setminus I} \uu \set{(a(\gs 1), \frac 1n) : \gs \in I} \uu \set {(a(\gs 2), \frac 1 n) : \gs \in I},\]
and using Equation~\eqref{eq341}, we have the constrained quantization error as 
\begin{align*}
V_n&= \frac 1{18}\Big(V_{n_1}+V_{n_2}-\frac 1 {n_1^2} - \frac 1{n_2^2}\Big)+\frac 1 {n^2}\\
&=\frac 1{18}\Big(\frac 1{18^{m-1}}\Big(2^{m}-n_1+\frac 1 9(n_1-2^{m-1})\Big) +\frac 1{18^{m-1}}\Big(2^{m}-n_2+\frac 1 9(n_2-2^{m-1})\Big)\Big)+\frac 1 {n^2}\\
&=\frac 1{18^{m}}\Big(2^{m+1}-n+\frac 1 9(n-2^{m})\Big)+\frac 1 {n^2}. 
\end{align*}
Thus, the theorem is true if $\ell(n)=m$. Hence, by the induction principle, the proof of the theorem is complete.
\end{proof}

 \section{Constrained quantization dimension and constrained quantization coefficient} \label{sec3}
 Let $P=\frac 12 P\circ T_1^{-1}+\frac 12 P\circ T_2^{-1}$ be a unique Borel probability measure on $\D R$ with support the Cantor set $C$ generated by
 the two contractive similarity mappings $T_1(x)=\frac 13 x$ and $T_2(x)=\frac 13 x+\frac 23$ for all $x\in \D R$.
 Since the Cantor set $ C$ under investigation satisfies the strong separation condition, with each $T_j$ having a contracting factor of $\frac{1}{3}$, the Hausdorff dimension of the Cantor set is equal to the similarity dimension.  Hence, from the equation $2(\frac 1{3})^\gb=1$,  we have $\dim_{\te{H}}(C)=\beta =\frac {\log 2}{\log 3}$.  It is known that the unconstrained quantization dimension of the probability measure $P$ exists and equals the Hausdorff dimension of the Cantor set (see \cite{GL2}). The work in this section shows that it is not true in the constrained case, i.e., the constrained quantization dimension $D(P)$ of the probability measure $P$, though it exists, is not necessarily equal to the Hausdorff dimension of the Cantor set. In this section, we show that the constrained quantization dimension $D(P)$ exists and equals one. We further show that the $D(P)$-dimensional constrained quantization coefficient exists as a finite positive number and equals $D(P)$, which is also not true in the unconstrained quantization for the Cantor distribution.

\begin{theorem}\label{theo2} 
The constrained quantization dimension $D(P)$ of the probability measure $P$ exists, and $D(P)=1$. 
\end{theorem}

\begin{proof}
For $n\in \D N$ with $n\geq 2$, let $\ell(n)$ be the unique natural number such that $2^{\ell(n)}\leq n<2^{\ell(n)+1}$. Then,
$V_{2^{\ell(n)+1}}\leq V_n\leq V_{2^{\ell(n)}}$. By Theorem \ref{theo1}, we see that $V_{2^{\ell(n)+1}}\to 0$ and  $V_{2^{\ell(n)}}\to 0 $ as $n\to \infty$, and so $V_n\to 0$ as $n\to \infty$,
i.e., $V_\infty=0$.
We can take $n$ large enough so that $ V_{2^{\ell(n)}}-V_{\infty}<1$. Then,
\[0<-\log (V_{2^{\ell(n)}}-V_{\infty})\leq -\log (V_n-V_\infty)\leq -\log (V_{2^{\ell(n)+1}}-V_\infty)\]
yielding
\[\frac{2\ell(n) \log 2}{-\log (V_{2^{\ell(n)+1}}-V_\infty)}\leq \frac{2\log n}{-\log(V_n-V_\infty)}\leq \frac{2(\ell(n)+1) \log 2}{-\log (V_{2^{\ell(n)}}-V_\infty)}.\]
Notice that
\begin{align*}
\lim_{n\to \infty} \frac{2\ell(n) \log 2}{-\log (V_{2^{\ell(n)+1}}-V_\infty)}&=\lim_{n\to \infty} \frac{2\ell(n) \log 2}{-\log(\frac{V}{9^{\ell(n)+1}}+\frac 1{4^{\ell(n)+1}})}=1, \te{ and }\\
\lim_{n\to \infty} \frac{2(\ell(n)+1) \log 2}{-\log (V_{2^{\ell(n)}}-V_\infty)}&=\lim_{n\to \infty} \frac{2(\ell(n)+1) \log 2}{-\log(\frac{V}{9^{\ell(n)}}+\frac 1{4^{\ell(n)}})}=1.
\end{align*}
Hence, $\lim_{n\to \infty}  \frac{2\log n}{-\log(V_n-V_\infty)}=1$, i.e., the constrained quantization dimension $D(P)$ of the probability measure $P$ exists and $D(P)=1$.
Thus, the proof of
the theorem is complete.
\end{proof}

\begin{theorem} \label{theo3} 
The $D(P)$-dimensional constrained quantization coefficient for $P$ is a finite positive number and equals the constrained quantization dimension $D(P)$.
\end{theorem}
\begin{proof}
For $n\in \D N$ with $n\geq 2$, let $\ell(n)$ be the unique natural number such that $2^{\ell(n)}\leq n<2^{\ell(n)+1}$.
Then, $0\leq 2^{\ell(n)+1} -n<2^{\ell(n)}$ and $0\leq \frac 19(n-2^{\ell(n)})<2^{\ell(n)}$. Hence, 
\[0\leq 2^{\ell(n)+1} -n +\frac 19(n-2^{\ell(n)})<2^{\ell(n)+1},\]
which implies
\begin{align*}
 0\leq \frac 1 {18^{\ell(n)}}V\Big(2^{\ell(n)+1}-n+\frac 1 9(n-2^{\ell(n)})\Big)<\frac{2^{\ell(n)+1}}{18^{\ell(n)}}
\end{align*}
yielding 
\[0\leq {n^2}  \Big(\frac 1{18^{\ell(n)}}V\Big(2^{\ell(n)+1}-n+\frac 1 9(n-2^{\ell(n)})\Big)\Big)<n^2\frac{2^{\ell(n)+1}}{18^{\ell(n)}}<\frac{2^{2\ell(n)+2}2^{\ell(n)+1}}{18^{\ell(n)}}=8\Big(\frac 4 9\Big)^{\ell(n)}.\]
Hence, by the squeeze theorem, we have 
\[\lim_{n\to \infty} {n^2}  \Big(\frac 1{18^{\ell(n)}}V\Big(2^{\ell(n)+1}-n+\frac 1 9(n-2^{\ell(n)})\Big)\Big)=0.\]
Again, as shown in the proof of Theorem \ref{theo2}, we have $V_\infty=\lim_{n\to \infty} V_n=0$. 
Thus, we deduce that 
\[\lim_{n\to \infty} n^2 (V_n-V_\infty)=1,\]
i.e., the  $D(P)$-dimensional constrained quantization coefficient for $P$ exists as a finite positive number and equals the constrained quantization dimension $D(P)$, which is the theorem. 
\end{proof} 

\section{Constrained quantization with some other families of constraints}
As defined in the previous sections, let $P=\frac 12 P\circ T_1^{-1}+\frac 12 P\circ T_2^{-1}$ be the unique Borel probability measure on $\D R$ with support the Cantor set $C$ generated by the two contractive similarity mappings $T_1(x)=\frac 13 x$ and $T_2(x)=\frac 13 x+\frac 23$ for all $x\in \D R$. In this section, in the following subsections, we give the optimal sets of $n$-points and the $n$th constrained quantization errors for different families of constraints. Then, for each family, we investigate the constrained quantization dimension and the constrained quantization coefficient.  

\subsection{Constrained quantization when the family is  
\texorpdfstring{$S_j=\set{(x, y) : 0\leq x\leq 1 \te{ and } y=1+\frac 1j}
$}{} for \texorpdfstring{$j\in \D N$}{}. }
Using the similar arguments as Lemma \ref{lemma0}, it can be shown that if $\ga_n$ is an optimal set of $n$-points for $P$, then $\ga_n\ci S_n$ for all $n\in \D N$.
Let us first state the following theorem, the proof of which is similar to Theorem~\ref{theo1}.

\begin{theo1}
 For $n\in \D N$ with $n\geq 2$ let $\ell(n)$ be the unique natural number with $2^{\ell(n)} \leq n<2^{\ell(n)+1}$. For $I\sci \set{1, 2}^{\ell(n)}$ with card$(I)=n-2^{\ell(n)}$ let $\ga_n(I)\ci S_n$ be the set such that
{\footnotesize\[\ga_n(I)=\set{(a(\gs), 1+\frac 1{n}) : \gs \in \set{1,2}^{\ell(n)} \setminus I} \uu \set{(a(\gs 1), 1+\frac 1 n) : \gs \in I} \uu \set {(a(\gs 2), 1+\frac  1n) : \gs \in I}.\]}
Then, $\ga_n:=\ga_n(I)$ forms an optimal set of $n$-points for $P$ with constrained quantization error 
\[V_n=\frac 1 {18^{\ell(n)}}V\Big(2^{\ell(n)+1}-n+\frac 1 9(n-2^{\ell(n)})\Big)+(1+\frac 1 n)^2,\]
where $V$ is the variance.
\end{theo1} 

\begin{remark1}
Notice that here $V_\infty=\mathop{\lim}\limits_{n\to \infty} V_n=1$. Thus, proceeding in the similar way as Theorem \ref{theo2} and Theorem \ref{theo3}, it can be seen that
\[\lim_{n\to \infty}  \frac{2\log n}{-\log(V_n-V_\infty)}=2 \te{ and } \lim_{n\to \infty} n^2 (V_n-V_\infty)=\infty,\]
i.e., the constrained quantization dimension $D(P)$ exists and equals $2$, but the constrained quantization coefficient does not exist. 
\end{remark1} 

\subsection{Constrained quantization when the family is  
\texorpdfstring{$S_j=\set{(x, y) : 0\leq x\leq 1 \te{ and } y=1}
$}{} for \texorpdfstring{$j\in \D N$}{}. }

Proceeding in a similar way as Theorem \ref{theo1}, we can show that the following theorem is true. 

\begin{theo1}
 For $n\in \D N$ with $n\geq 2$ let $\ell(n)$ be the unique natural number with $2^{\ell(n)} \leq n<2^{\ell(n)+1}$. For $I\sci \set{1, 2}^{\ell(n)}$ with card$(I)=n-2^{\ell(n)}$ let $\ga_n(I)\ci S_n$ be the set such that
\[\ga_n(I)=\set{(a(\gs), 1) : \gs \in \set{1,2}^{\ell(n)} \setminus I} \uu \set{(a(\gs 1), 1 ) : \gs \in I} \uu \set {(a(\gs 2), 1) : \gs \in I}.\]
Then, $\ga_n:=\ga_n(I)$ forms an optimal set of $n$-points for $P$ with constrained quantization error 
\[V_n=\frac 1 {18^{\ell(n)}}V\Big(2^{\ell(n)+1}-n+\frac 1 9(n-2^{\ell(n)})\Big)+1,\]
where $V$ is the variance.
\end{theo1} 
 
Notice that here $V_\infty=\mathop{\lim}\limits_{n\to \infty} V_n=1$. If $\gb$ is the Hausdorff dimension of the Cantor set $\C C$, then, $\gb=\frac{\log 2}{\log 3}$. Then, the following lemma and theorems are equivalent to the lemma and theorems that appear in the last section in \cite{GL2}. For the proofs, one can consult \cite{GL2}. 

\begin{theo1}
The set of accumulation points of the sequence $\Big(n^{\frac 2{\gb}} (V_n-V_\infty)\Big)_{n\in \D N}$ equals 
\[\Big[V, f\Big(\frac {17}{8+4\gb}\Big)\Big],\]
where $f : [1, 2] \to \D R$ is such that $f(x)=\frac 1{72} x^{\frac 2\gb}(17-8x)$. 
\end{theo1} 

\begin{lemma1}
Let $n\in \D N$. Then, \[\frac 1{72} \leq n^{\frac 2{\gb}} (V_n-V_\infty)\leq \frac 9 8.\]
\end{lemma1} 
\begin{theo1}
The constrained quantization dimension of $P$ equals the Hausdorff dimension $\gb:=\frac{\log 2}{\log 3}$
of the Cantor set, i.e., 
\[D(P)=\lim_{n\to \infty}  \frac{2\log n}{-\log(V_n-V_\infty)}=\gb.\]
\end{theo1} 

\begin{remark1}
Thus, in this case, we see that the constrained quantization dimension exists and equals the Hausdorff dimension $\gb$ of the Cantor set, but the constrained quantization coefficient does not exist. 
\end{remark1}

\subsection{Constrained quantization when the family is  
\texorpdfstring{$S_j=\set{(x, y) : 0\leq x\leq 1 \te{ and } y=1-\frac 1j}$}{} for \texorpdfstring{$j\in \D N$}{}. }
Using the similar arguments as Lemma \ref{lemma0}, it can be shown that if $\ga_n$ is an optimal set of $n$-points for $P$, then $\ga_n\ci S_1$ for all $n\in \D N$.
Let us first state the following theorem, the proof of which is similar to Theorem \ref{theo1}.

\begin{theo1} \label{theo34} 
 For $n\in \D N$ with $n\geq 2$ let $\ell(n)$ be the unique natural number with $2^{\ell(n)} \leq n<2^{\ell(n)+1}$. For $I\sci \set{1, 2}^{\ell(n)}$ with card$(I)=n-2^{\ell(n)}$ let $\ga_n(I)\ci S_1$ be the set such that
\[\ga_n(I)=\set{(a(\gs), 0) : \gs \in \set{1,2}^{\ell(n)} \setminus I} \uu \set{(a(\gs 1), 0) : \gs \in I} \uu \set {(a(\gs 2), 0) : \gs \in I}.\]
Then, $\ga_n:=\ga_n(I)$ forms an optimal set of $n$-points for $P$ with constrained quantization error 
\[V_n=\frac 1 {18^{\ell(n)}}V\Big(2^{\ell(n)+1}-n+\frac 1 9(n-2^{\ell(n)})\Big),\]
where $V$ is the variance.
\end{theo1} 

\begin{remark1}
By Theorem \ref{theo34}, we see that for the family $S_j=\set{(x, y) : 0\leq x\leq 1 \te{ and } y=1-\frac 1j}$, where $j\in \D N$, the optimal sets of $n$-points and the corresponding constrained quantization error $V_n$ coincide with the optimal sets of $n$-means and the corresponding quantization error for the Cantor distribution $P$ that occurs in \cite{GL2}. Thus, all the results in \cite{GL2} are also true here. 
\end{remark1}

\subsection*{Acknowledgement}
The first author is grateful to her supervisor \textit{Professor Tanmoy Som} of the IIT(BHU), Varanasi, India, for his constant support in preparing this manuscript.

\section*{Declaration}

\noindent
\textbf{Conflicts of interest.} We do not have any conflict of interest.\\
\\
\noindent
\textbf{Data availability:} No data were used to support this study.\\
\\
\noindent
\textbf{Code availability:} Not applicable\\
\\
\noindent
\textbf{Authors' contributions:} Each author contributed equally to this manuscript.

 \bibliographystyle{unsrt}
  \bibliography{References}

\end{document}